\documentclass[12pt]{amsart}
\usepackage{amssymb}
\usepackage{import}
\usepackage{tikz}
\usetikzlibrary{shapes.geometric}
\usetikzlibrary{positioning}
\newtheorem{theorem}{Theorem}[section]

\newtheorem{lemma}{Lemma}[section]
\newtheorem{corollary}{Corollary}[section]

\usepackage{thmtools}
\usepackage{thm-restate}

\numberwithin{mytheorem}{section} 

\numberwithin{mylemma}{section} 

\numberwithin{mycorollary}{section} 
\usepackage{hyperref}

\usepackage{cleveref}

\usepackage{color}

\usepackage{hyperref}

\title[Proof of TAC for GMR]{The Graph Minor Relation Satisfies the Twin Alternative Conjecture}

\author{Jorge Bruno \\ email: \href{mailto:jorgebruno@bpp.com}{jorgebruno@bpp.com}}

\newcommand{\N}{\mathbb{N}}

\begin{document} 
\maketitle

\begin{abstract} In 2006 Bonato and Tardif posed the Tree Alternative Conjecture (TAC): the equivalence class of a tree under the embeddability relation is, up to isomorphism, either trivial or infinite. In 2022 Abdi, et al. provided a rigorous exposition of a counter-example to TAC developed by Tetano in his 2008 PhD thesis and in 2023 the present author provided a positive answer to TAC for the topological minor relation. Along with embeddability and the topological minor, the graph minor relation completes the triad of the most widely studied graph relations. By means of exploiting some basic set theoretic notions, in this paper we provide a positive answer to TAC for the graph minor relation.  

\end{abstract}
\section{Introduction} 
Consider the following local operations on a graph \cite{F}:
\begin{enumerate}
\item Edge removal  (\textcolor{red}{$\backslash e$}): 
\begin{center}
 \includegraphics[scale = 1.2]{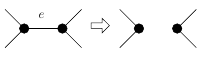} 
 \end{center}
\item Edge contraction  (\textcolor{red}{$\slash e$}): 
\begin{center}
 \includegraphics[scale = 1.2]{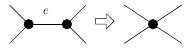} 
 \end{center}
 \item Vertex removal  (\textcolor{red}{$\backslash v$}): 
\begin{center}
 \includegraphics[scale = 1.2]{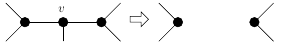} 
 \end{center}

 \item Deg-2 vertex dissolution  (\textcolor{red}{$\slash v$}): 
\begin{center}
 \includegraphics[scale = 1.2]{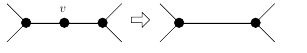} 
 \end{center}
 \end{enumerate}
For $\mathcal{L} \subseteq \{ \textcolor{red}{\backslash e, \slash e,  \backslash v, \slash v}\}$ then $H\leq_\mathcal{L} G$ if $H$ is obtained from $G$ by a sequence of operations from $\mathcal{L} $. The table below summarises the most commonly studied graph relations arising from the above local graph operations.\footnote{Since all trees are connected, the induced and regular versions of the above are equivalent. Therefore, given that our results concern only trees, we make no distinction between the two versions of any of the above graph relations.} 

\bigskip

\begin{center}
\begin{tabular}{ | c || c | c | c | c | }
\hline
Relation & $ \textcolor{red}{\backslash e}$ &  $\textcolor{red}{\slash e}$ & $ \textcolor{red}{\backslash v}$ &  $\textcolor{red}{\slash v}$\\
\hline
\hline
 subgraph/embeddable ($\leq$) &\textcolor{blue}{ $\bullet$} & &\textcolor{blue}{ $\bullet$} & \\ 
 \hline
  induced subgraph/strongly embeddable & & &\textcolor{blue}{ $\bullet$} & \\ 
\hline
  \hline
 topological minor ($\leq^\sharp$) & \textcolor{blue}{ $\bullet$}& & \textcolor{blue}{ $\bullet$}&\textcolor{blue}{ $\bullet$} \\ 
 \hline
 induced topological minor & & &\textcolor{blue}{ $\bullet$} &\textcolor{blue}{ $\bullet$} \\ 
\hline
 \hline
 graph minor ($\leq^*$) & \textcolor{blue}{ $\bullet$}&\textcolor{blue}{ $\bullet$} &\textcolor{blue}{ $\bullet$} & \\ 
 \hline
   induced graph minor && \textcolor{blue}{ $\bullet$}& \textcolor{blue}{ $\bullet$} & \\ 
   \hline

\end{tabular}
\end{center}

\bigskip

While the local operations defined above intuitively describe the various minor relations for finite graphs, extending them to infinite graphs via potentially infinite sequences of operations requires strict formalisation. For instance, performing an infinite sequence of deg-2 vertex dissolutions (or edge contractions) along a ray could theoretically result in ill-defined structures, such as edges with only one endpoint or none at all. To avoid these pathological limits and rigorously define $\leq^\sharp$ and $\leq^*$ for infinite trees, we employ the standard model-theoretic approach. Rather than relying on sequences of local operations, we define the minor relations via a mapping that assigns to each vertex of the minor a connected subset (or subtree) of the host graph. This ensures that all vertices and adjacencies remain well-defined regardless of the tree's cardinality. We formally introduce this model definition in Section 2.

\bigskip

\noindent
Two trees $T,S$ are {\it mutually embeddable}, $T\equiv S$, (resp. {\it topologically equivalent}, $T\equiv^\sharp S$, and {\it mutual minors}, $T\equiv^* S$) if $T \leq S \leq T$ (resp. $T \leq^\sharp S \leq^\sharp T$ and $T \leq^* S \leq^* T$). The symbol $\cong$ is reserved for the isomorphism relation. Since any deg-2 vertex dissolution can be interpreted as an edge dissolution we get that 
\[
T\cong S \Rightarrow T \leq S \Rightarrow T \leq^\sharp S \Rightarrow T \leq^* S.
\]
 The equivalence class of a tree $T$ with respect to $\equiv, \equiv^\sharp$ and $\equiv^*$ is denoted by $[T]$, $[T]_\sharp$ and $[T]_*$, respectively. 

The authors of \cite{A} proved that $|[T]| = 1$ or $\infty$ for any rayless tree and conjectured the same must be true of any tree: the Tree Alternative Conjecture (TAC) stated that the number of isomorphism classes of trees mutually embeddable with a given tree $T$ is either 1 or infinite. TAC is certainly true for finite trees (since mutually embeddable finite trees are necessarily isomorphic) and has been confirmed for a number of nontrivial classes of infinite trees -  \cite{A}, \cite{LPS} and \cite{T}. In particular, in \cite{T} it is shown that TAC holds for all rooted trees, where a {\it rooted tree} $(T,r)$ is one with a distinguished vertex $r\in V(T)$. In 2022 a counterexample to TAC was found by LaFlamme et al. \cite{ALTW} where for each $n\in \N$ an unrooted and locally finite tree $T_n$ is constructed with $|[T_n]| = n$. 

A natural way to generalise this line of enquiry is to extend TAC to the topological and graph minor relations. Given $\hat{R} \in \{\equiv, \equiv^\sharp, \equiv^*\}$ we define\\ 

\noindent
  {\bf TAC($\hat{R}$).} For any tree $T$,
  
   \[
   |[T]_{\hat{R}}| =1 \text{ or } \geq \aleph_0.
   \]  
   
   \bigskip
  \noindent
    {\bf RootedTAC($\hat{R}$).} For any tree $T$ and any $r\in V(T)$,
   \[
   |[(T,r)]_{\hat{R}}| = 1 \text{ or } \geq \aleph_0.
   \]  
   
   \bigskip
  \noindent
As stated earlier, TAC($\equiv$) is false and RootedTAC($\equiv$) is true. RootedTAC($\equiv^\sharp$) and TAC($\equiv^\sharp)$ are true (\cite{BP1}, \cite{BP2}) and in this manuscript we prove the following result.\\

  \noindent
  {\bf Main Theorem.} 
 TAC($ \equiv^*$) and RootedTAC($ \equiv^*$) hold. \\
 
 \noindent
 The table below summarises past and present developments.

  \begin{table}[h!]
\centering
\begin{tabular}{|l|c|c|}
\hline
\textbf{Relation ($\hat{R}$)} & \textbf{RootedTAC($\hat{R}$)} & \textbf{TAC($\hat{R}$)} \\ \hline
Embeddability ($\equiv$) & \begin{tabular}[c]{@{}c@{}}True\\ (Tyomkyn, 2008)\end{tabular} & \begin{tabular}[c]{@{}c@{}}False\\ (Abdi et al., 2022)\end{tabular} \\ \hline
Topological Minor ($\equiv^\sharp$) & \begin{tabular}[c]{@{}c@{}}True\\ (Bruno \& Szeptycki, 2022)\end{tabular} & \begin{tabular}[c]{@{}c@{}}True\\ (Bruno \& Szeptycki, 2022)\end{tabular} \\ \hline
Graph Minor ($\equiv^*$) & \begin{tabular}[c]{@{}c@{}}\textbf{True}\\ \textbf{(This Paper)}\end{tabular} & \begin{tabular}[c]{@{}c@{}}\textbf{True}\\ \textbf{(This Paper)}\end{tabular} \\ \hline 

\end{tabular}
\label{tab:tac_results}
\end{table}
 
Our approach to proving the Main Theorem is to deal with large and small trees separately: a {\it small tree} is one where every ray is eventually bare and a ray $R = v_1\ldots$ is {\it eventually bare} provided $\exists M \in \N$ where, for all $n\geq M$, $\deg(v_n) = 2$. Trees that are not small are said to be {\it large}. Since all rayless trees are small, the following result holds.\\

\begin{corollary} For any $\hat{R} \in \{\equiv, \equiv^\sharp, \equiv^*\}$, RootedTAC($\hat{R})$ holds. Further, when considering only small trees, TAC($\hat{R})$ also holds.
 \end{corollary}

It is important to note that the case for large trees under the graph minor relation follows directly from the author's previous work on the topological minor relation \cite{BP1}. That prior work utilised the method of `curtailing' to establish that the equivalence classes of large trees have size at least $2^{\aleph_0}$. Because the topological minor relation is strictly stronger than the graph minor relation, the large tree case is immediately satisfied. Consequently, the primary contribution and technical focus of the present paper lie entirely in establishing the conjecture for {\it small trees} under the graph minor relation. \section{Background}

An equivalent interpretation for the graph minor relation employs the notion of connected subsets: $T \leq^* S$ precisely when there exists a function $\mu: V(T) \to \text{Sub}(S)$, where $\text{Sub}(S)$ is the collection of all connected subgraphs (i.e., subtrees) of $S$, so that 
\begin{enumerate}
\item $V(\mu(v)) \cap V(\mu(w)) = \emptyset$ whenever $v\not = w$, and
\medskip
\item $\exists u_vu_w\in E(S)$ with $u_v \in V(\mu(v))$ and $u_w\in V(\mu(w))$ whenever $vw \in E(T)$.
\end{enumerate}
We call the function $\mu$ a {\it model of $T$ in $S$}. Composing models is done as follows: given models $\mu_1: V(T) \to \text{Sub}(S_1)$ and $\mu_2:V(S_1) \to \text{Sub}(S_2)$, we {\it compose} $\mu_1$ with $\mu_2$ as $\mu_2\circ^*\mu_1: V(T) \to \text{Sub}(S_2)$, where $(\mu_2\circ^*\mu_1)(v)$ is the subtree of $S_2$ spanned by all subtrees $\mu_2(w)$ for $w\in V(\mu_1(v))$ and the edges joining them.

A {\it rooted tree} $(T,r)$ is one with a distinguished vertex $r\in V(T)$. All rooted trees naturally give rise to a tree order on its set of vertices: $v\leq_T w$ precisely when the unique path from $r$ to $w$ contains $v$. Given a vertex $v\in V(T)$, the full subtree of $(T,r)$ rooted at $v$ is the one containing exactly those $w\geq_T v$ and is denoted by $(T_v,v)$. Any rooted tree $(T,r)$ can be equivalently interpreted as a meet semilattice - $(V(T),r,\wedge_T)$ - and, consequently, any connected component in $(T,r)$ must have a minimum vertex in $T$'s order. In turn, the notion of graph minor can be naturally extended to rooted trees as follows. A rooted tree $(T,r)$ is a {\it rooted graph minor} of $(S,s)$, $(T,r) \leq^*(S,s)$, if there exists a model $\mu:V(T) \to \text{Sub}(S)$ that preserves the meet semilattice structure of $(T,r)$ as follows: for any pair $v,w\in V(T)$ we have\footnote{The stronger condition $\bigwedge_S\{t\in V(\mu(v)) \cup V(\mu(w))\}= \bigwedge_S\mu(w\wedge_T v)$ describes the topological minor relation.}  
\[
\bigwedge_S\left\{t\in V(\mu(v)) \cup V(\mu(w))\right\} \in \mu(w\wedge_T v).
\]
Equivalently, $(T,r) \leq^*(S,s)$ if there exists a model $\mu:V(T) \to \text{Sub}(S)$ so that the function $\mu_{\min}:V(T) \to V(S)$ with $v \mapsto \min(\mu(v))$ establishes a meet semilattice homomorphism between $(V(T), r, \wedge_T)$ and $(\mu_{\min}(V(T)), \mu_{\min}(r), \wedge_S)$. The equivalence class of a tree $T$ with respect to the graph minor relation is denoted by $[T]_*$ and if the tree is rooted then we write $[(T,r)]_*$.



\subsection{Large Trees} Our approach to establishing the Main Theorem is to deal with large and small trees separately. In \cite{BP1} it is shown that for any large tree $T$ and any $r\in V(T)$ both $|[T]_\sharp|$ and $|[(T,r)]_\sharp| \geq 2^{\aleph_0}$. Since $\leq^\sharp$ is stronger than $\leq^*$, it follows that the Main Theorem holds for all large trees.

\begin{theorem}\label{thm:large} TAC($\equiv^*$) and RootedTAC($\equiv^*$) hold for all large trees.
\end{theorem}

\section{Small Trees}

We turn our attention to proving the Main Theorem for small trees: in Section~\ref{sec:rootedSmall} we establish it for rooted trees and then extend it to all trees in Section~\ref{sec:unrootedSmall}. First, we focus on some basic preliminary facts. Given a tree $T$, let $F(T) = \{v\in V(T)\mid \deg(v)>2\}$ and let $\text{inf}(T) = \{v\in V(T) \mid \deg(v) = \infty\}$. If $T$ is small and locally finite, then clearly $F(T)$ must be finite.

\begin{lemma}\label{lem:basics} Let $T,S$ be trees and $\mu:V(T) \to \text{Sub}(S)$ be a model of $T$ in $S$. It follows that 
\begin{enumerate}
\item if $S$ is small then so is $T$,
\item if $v\in F(T)$ then $V(\mu(v)) \cap F(S) \not = \emptyset$, and
\medskip
\item if $v\in \text{inf}(T)$ then $V(\mu(v)) \cap \text{inf}(S) \not = \emptyset$, or there exists a ray in $S$ that is not eventually bare and is contained in $\mu(v)$.
\end{enumerate}
\end{lemma}
\begin{proof} (1) is obvious, and were (2) false, then $\mu(v)$ would be a path and could be connected to at most two other subtrees of $S$.\\

For (3), take any $v\in \text{inf}(T)$ and let $u_i$, $i\in \N$, be any infinite collection of distinct neighbours of $v$. Because $\mu:V(T) \to \text{Sub}(S)$ is a model of $T$ in $S$, there must exist an infinite sequence of distinct edges $w_iv_i \in E(S)$ with $v_i \in V(\mu(v))$ and $w_i\in V(\mu(u_i))$. Therefore, if $V(\mu(v))\cap \text{inf}(S)=\emptyset$ then, as a direct consequence of König's Lemma, the closure of the vertices $\{v_i \mid i \in \N\}$ (i.e., the union of all paths connecting them) spans an infinite, locally finite subtree of $\mu(v)$ . Consequently, this subtree must contain a ray $R$ in $S$ that is not eventually bare.
\end{proof}

Given a collection $A \subseteq V(T)$ we denote $\overline{A}$ as the {\it closure of $A$ in $T$}: $\overline{A}$ is $A$ along with all paths in $T$ joining any pair of vertices in $A$. The following result is key in establishing the Main Theorem for small trees. Briefly, it states that all four equivalence relations (i.e., $\cong, \equiv,\equiv^\sharp , \equiv^*$) coincide on the collection of small locally finite trees.

\begin{theorem}\label{thm:equiv} Let $T,S$ be locally finite and small. It follows that $T\cong S \Leftrightarrow T\equiv^*S$. 
\end{theorem}
\begin{proof} 
The authors of \cite{BP1} prove that $T\cong S \Leftrightarrow T\equiv^\sharp S$ (see Theorem 5). In turn, we focus on showing that $T\equiv^*S$ implies $T\equiv^\sharp S$. Suppose that $\mu: V(T) \to Sub(S)$ witnesses $T \leq^* S$ and recall that by Lemma~\ref{lem:basics} for each $v\in F(T)$ it must be that $V(\mu(v))\cap F(S) \not = \emptyset$. Hence, $F(T) \leq F(S)$ and since $T \geq^* S$ we also get that $F(S) \leq F(T)$. It follows that $V(\mu(v))\cap F(S) = \{w\}$ for any $v\in F(T)$ since $|F(S)| = |F(T)|$. It then follows that any edge contraction employed to turn $T$ into $S$ must be performed on an edge $vw \in E(T)$ with both $\deg(v), \deg(w) \leq 2$. Any such edge contraction is a deg-2 vertex dissolution. Hence, $T \leq^\sharp S$ and, by symmetry, $T \geq^\sharp S$. 
\end{proof}

\subsection{RootedTAC($\equiv^*$)} \label{sec:rootedSmall}
 For a rooted tree $(T,r)$ we let $succ(v) $ denote the collection of immediate successors of a vertex $v$ and $(T_v,v)$ denote the full subtree of $(T,r)$ generated by all $w\geq_T v$.

\begin{lemma} For any rooted tree $(T,r)$ and any $v\in V(T)$ we have $|[(T_v,v)]_*|\leq |[(T,r)]_*|$ . 
\end{lemma}
\begin{proof} It suffices to prove the claim for any vertex in $succ(r)$. Hence, fix a $v\in succ(r)$ and take any $(S,s) \in [(T_v,v)]_*$. Construct the tree $(T',r)$ as follows: for any $w\in succ(v)$
\begin{itemize}
\item replace $(T_w,w)$ with $(S,s)$ if $(T_w,w) \equiv^* (S,s)$, and
\item leave $(T_w,w)$ as it is, otherwise.
\end{itemize}
By design, $(T',r) \equiv^* (T,r)$ and $(T',r) \not \cong (T,r)$.
\end{proof}

In view of the preceding lemma, assume for contradiction that there exists a small tree $(T,r)$ with a finite, non-trivial equivalence class (i.e., $1 < |[(T,r)]_*| < \aleph_0$). Suppose that every vertex $v$ with $1 < |[(T_v, v)]_*| < \aleph_0$ has at least one successor $w \in succ(v)$ such that $1 < |[(T_w, w)]_*| < \aleph_0$. We could then greedily choose these successors to construct an infinite increasing sequence $r <_T v_1 <_T v_2 <_T \ldots$ where every vertex roots a subtree with an equivalence class strictly greater than 1. However, because $T$ is a small tree, any infinite increasing sequence of vertices must span an eventually bare ray. Consequently, the sequence of subtree equivalence class sizes along any ray must eventually reduce to a tail of 1s (as bare paths have a trivial equivalence class). This contradicts the existence of our infinite sequence. Therefore, the chain must terminate: there must exist some vertex $u \geq_T r$ such that $1 < |[(T_u, u)]_*| < \aleph_0$, but for all $w \in succ(u)$, $|[(T_w, w)]_*| = 1$.In light of this observation, we finish this section by proving that this configuration is strictly impossible, thereby proving that no small tree can have a finite, non-trivial equivalence class.

\begin{theorem} For any small tree $(T,r)$ with $|[(T_v, v)]_*| =1$, $\forall v \in succ(r)$, it follows that $|[(T,r)]_*| =1$ or $\infty$. 
\end{theorem}

\begin{proof} Let $(T,r)$ and $(S,s)$ be two non-isomorphic trees with $(S,s)\in [(T,r)]_*$, and models $\mu_T: V(T) \to \text{Sub}(S)$ and $\mu_S: V(S) \to \text{Sub}(T)$. Further assume that $|[(T_v, v)]_*| =1$ for any $v \in succ(r)$. We seek to prove that $|[(T,r)]_*| \geq \aleph_0$. 

Put $\mathcal{I}= \{(X_\alpha,\alpha) \mid \alpha\in \kappa\}$  as the set of all isomorphism types of all rooted trees $(T_v,v)$ and $(S_w,w)$ with $v\in succ(r)$  and $w\in succ(s)$. Define $f_T: \{(T_v,v)\mid v\in succ(r)\} \to \kappa$ so that $f_T(T_v,v) = \alpha$ provided $(T_v,v) \cong (X_\alpha, \alpha)$ (resp. $f_S: \{(S_w,w)\mid w \in succ(s)\} \to \kappa$ so that $f_S(S_w, w) = \alpha$ provided $(S_w, w) \cong (X_\alpha, \alpha)$). Since $(T,r) \not \cong (S,s)$ there must exist a $\beta \in \kappa$ for which, without loss of generality, $ |f_T^{-1}(\beta)| < |f_S^{-1}(\beta)|$. Put $T_B = f_T^{-1}(\beta)$ and $S_B = f_S^{-1}(\beta)$, and $\eta = |f_T^{-1}(\beta)|$ and $\zeta = |f_S^{-1}(\beta)|$. Enumerate the elements of $T_B$ and $S_B$ as $(T_i,i)$ and $(S_j,j)$ with $i\in \eta$ and $j\in \zeta$. We are then left with the following complementary scenarios.\\

\noindent
{\bf CASE 1:} $\forall j \in \zeta$, $\exists i_j \in \eta$ such that $V(\mu_S(S_j)) \subseteq V(T_{i_j})$\\ 

\noindent
{\bf CASE 2:} $\exists j \in \zeta$ with $V(\mu_S(S_\gamma)) \subseteq V(T_v)$, $(T_v, v) \not \in f_T^{-1}(\beta)$. \\

Before investigating the above cases we establish a simple yet useful result.

\begin{lemma}\label{lem:root_anchor} Let $(T_1, r_1)$ and $(T_2, r_2)$ be isomorphic small rooted trees. If $\mu: V(T_1) \to \text{Sub}(T_2)$ is a model witnessing $(T_1, r_1) \leq^* (T_2, r_2)$ such that $r_2 \notin V(\mu(r_1))$, then $(T_2, r_2)$ must be a bare ray or a finite path.\end{lemma}\begin{proof}Suppose $r_2 \notin V(\mu(r_1))$. Because $V(\mu(r_1))$ is a connected subset that does not contain the root $r_2$, the entire image of $T_1$ under $\mu$ must be mapped strictly "above" $r_2$, meaning it is contained within some proper principal subtree $(T_{2,v}, v)$ for a successor $v \in succ(r_2)$. Because $(T_1, r_1) \cong (T_2, r_2)$, this embedding implies that $(T_2, r_2)$ can be minor-embedded into a proper principal subtree of itself. Iterating this embedding generates an infinite sequence of vertices $r_2 <_T v = x_1 <_T x_2 <_T \ldots$, where each $x_n$ roots a nested isomorphic copy of the tree. This strictly increasing sequence of vertices spans a ray $R$ in $T_2$. Clearly, $R$ is a bare ray precisely when $(T_2, r_2)$ is a bare ray.\end{proof}


\noindent
{\bf CASE 1:}
Observe that $\zeta = \sum_{i \in \eta} |A_i|$, where$$A_i = \{(S_j,j) \in f_S^{-1}(\beta) \mid V(\mu_S(S_j)) \subseteq V(T_i)\}.$$Because $\eta < \zeta$ there must exist at least one $\gamma \in \eta$ such that $|A_\gamma| \geq 2$. Recall that $(T_\gamma, \gamma) \cong (S_j, j)$ and for all $j \in \zeta$. The fact that $|A_\gamma| \geq 2$ implies that the subtree $(T_\gamma, \gamma)$ contains at least two disjoint models of itself. We can then apply Theorem~\ref{lem:root_anchor} to obtain a contradiction.\\

\noindent
{\bf CASE 2:} First, let $(T', r)$ be the rooted tree obtained from $(T, r)$ by deleting all $\eta$ principal subtrees that are isomorphic to $(S_\gamma, \gamma)$. That is, all trees isomorphic to $(S_\gamma, \gamma)$ rooted at a successor or $r$, along with their root and, thus, edge joining them to $r$. For any cardinal $\lambda$, we define the tree $(T'_\lambda, r)$ as the tree $(T', r)$ with exactly $\lambda$ new, disjoint copies of $(S_\gamma, \gamma)$ attached to the root $r$.\\

First consider the case where $\eta$ is finite (tacitly including the case where $\eta = 0$). Fix a $\gamma \in \zeta$ for which $\mu(S_\gamma,\gamma) \subseteq (T_v,v)\not \in f_T^{-1}(\beta)$. Next, for each $n\in \omega$, put $w_n \in succ(s)$ so that $(\mu_T\circ^*\mu_S)^n[(S_\gamma,\gamma)] \subseteq (S_{w_n}, w_n)$ and $v_n\in succ(r)$ with $\mu_S\circ^*(\mu_T\circ^*\mu_S)^n[(S_\gamma,\gamma)] \subseteq (T_{v_n}, v_n)$. Clearly, $(S_{w_0},w_0) = (S_\gamma,\gamma)$ and $(T_{v_0}, v_0) = (T_v, v)$, and 
\[
(S_{w_0},w_0) \leq^* (T_{v_0},v_0) \leq^*(S_{w_1},w_1) \leq^*(T_{v_1},v_1) \leq^* \ldots
\]
 Observe that $w_n \not = w_0$ for any $n\in \omega$ since then $(S_\gamma, \gamma) \cong (T_v,v)$ from the assumption that $|[(T_u, u)]_*| =1$ for any $u \in succ(r)$. In general, we claim that the sequences $(w_n)_{n \in \omega}$ and $(v_n)_{n \in \omega}$ consist of strictly non-repeating elements. Suppose for contradiction that the sequence loops. Without loss of generality, assume $w_1 = w_2$. The resulting mutual embedding forces $(S_{w_1}, w_1) \equiv^* (T_{v_1}, v_1)$. Under our assumption that $|[(T_u, u)]_*| = 1$ for all immediate successors, this equivalence guarantees strict isomorphism: $(S_{w_1}, w_1) \cong (T_{v_1}, v_1)$. Theorem~\ref{lem:root_anchor} states that $v_1 \in V(\mu_T(w_1))$. Since $\mu_T$ also maps $(T_{v_0},v_0)$ into $(S_{w_1}, w_1)$ and $v_0$, $v_1$ are incomparable in $T$, then $\mu_T$ would fail to preserve the meet semilattice order of $T$. Thus, we get two sequences of non-repeating elements: $(w_n)_{n\in\omega}$ and $(v_n)_{n\in\omega}$. \\

\noindent
We finish the case where $\eta$ is finite as follows. Clearly, $(T'_{\lambda}, r) \geq^* (T,r)$ for each $\lambda \geq \eta$. Working in the other direction, given any finite $\lambda \geq \eta$, let $u_1, \dots, u_\lambda \in succ(r)$ be the roots of the $\lambda$ new principal subtrees in $T'_\lambda$ that are isomorphic to $(S_\gamma, \gamma)$. We define the model $\mu_\lambda: V(T'_{\lambda}) \to \text{Sub}(T)$ piecewise as follows:
\begin{itemize}
\item For each new copy $1 \leq l \leq \lambda$, let 

\[
\mu_\lambda \restriction T'_{u_l} = \mu_S \circ^* (\mu_T \circ^* \mu_S)^{l-1},
\]

\item For the original chain branches $n \in \omega$, let 

\[
\mu_\lambda \restriction T_{v_n} = (\mu_S \circ^* \mu_T)^{\lambda},
\]

\item Let $\mu_\lambda$ act as the identity mapping everywhere else in $T'_\lambda$.
\end{itemize}

\noindent
Since $(T'_{u_l}, u_l) \cong (S_\gamma, \gamma)$, this mapping embeds the $\lambda$ new copies into the first $\lambda$ slots of the chain, while the operator $(\mu_S \circ^* \mu_T)^\lambda$ shifts the original branches downward without collision. Thus, $\mu_\lambda$ is a valid model.

The case for $\eta$ infinite follows a similar flavour to the finite case; we prove that $(T', r) \equiv^* (T,r)$ (and thus, $(T'_\lambda, r) \equiv^* (T,r)$ for all cardinals $\lambda \leq \eta$). Clearly, we need only show $(T, r) \leq^* (T',r)$. Since $\zeta > \eta$ (and both are infinite), then there exists a subset $S'_B \subset S_B$ with $|S'_B| = \zeta$ such that for every $(S_w, w) \in S'_B$, its image under $\mu_S$ is entirely contained within a strictly non-$\beta$-type principal subtree of $T$. Because $|T_B| = \eta < \zeta = |S'_B|$, there exists an injection $m: T_B \to S'_B$. Let $N_{recv}$ denote the set of non-$\beta$-type principal subtrees $(T_u, u)$ in $T$ that contain the image $\mu_S(m(T_v, v))$ for some $(T_v, v) \in T_B$. 

Finally, we define the model $\mu: V(T) \to \text{Sub}(T')$ piecewise as follows:
\begin{enumerate}
\item For each $(T_v, v) \in T_B$, let $\phi_v: T_v \to m(T_v)$ be an isomorphism. We set $\mu \restriction T_v = \mu_S \circ^* \phi_v$.
\item For each receiving non-$\beta$-type branch $(T_u, u) \in N_{recv}$, we evacuate the native infrastructure by setting $\mu \restriction T_u = (\mu_S \circ^* \mu_T)^2 \restriction T_u$.
\item Let $\mu$ act as the identity mapping on all remaining vertices of $(T, r)$.
\end{enumerate}

For the exact same reasons outlined in the finite case ($\eta \in \omega$), the shifted native infrastructure in Part (2) strictly does not clash with the incoming $\beta$-copies mapped in Part (1). This establishes that $\mu$ is a valid graph minor model, yielding $(T, r) \leq^* (T', r)$ and completing the proof.

\end{proof}
\subsection{Unrooted}\label{sec:unrootedSmall}
We are now ready to establish Theorem 1 for all unrooted trees. We do this by proving Theorem~\ref{thm:fixed} as an adaptation of a fixed-point theorem of Polat and Sabidussi from \cite{P}. To that end, we make use of Lemma 1.1 from \cite{P}.

\begin{lemma}\label{lem:Polat}[Lemma 1.1, \cite{P}] Let $G$ be a rayless graph and $(A_\alpha)_{\alpha\in \text{Ord}}$ a decreasing sequence of subsets of $V(G)$ such that:
\begin{enumerate}
\item $A_\alpha = \bigcap_{\beta < \alpha} A_\beta$ for limit ordinals $\alpha$, and
\item each $A_\alpha$ induces a connected subgraph of $G$.
\end{enumerate}
Let $\kappa$ denote the smallest ordinal for which the sequence $(A_\alpha)_{\alpha\in \text{Ord}}$ becomes constant. If $A_\kappa =\emptyset$, then $\kappa$ is a successor ordinal.
\end{lemma}

Let $\text{GME}(T)$ denote the collection of all graph minor models of $T$ in $T$, and let $\mu \in \text{GME}(T)$. By Lemma~\ref{lem:basics}, since $T$ is small, for any $v\in \text{inf}(T)$, the branch set $\mu(v)$ contains at least one vertex of infinite degree. While $\mu(v)$ may also contain vertices of finite degree lying outside of $\overline{\text{inf}(T)}$, the intersection of their vertex sets, $V(\mu(v)) \cap V(\overline{\text{inf}(T)})$, induces a non-empty, connected subgraph in $T$. Furthermore, any edge connecting adjacent branch sets must lie entirely within $\overline{\text{inf}(T)}$ to maintain the connectivity of their union. It follows that the restricted mapping $v \mapsto T[\,V(\mu(v)) \cap V(\overline{\text{inf}(T)})\,]$ (i.e., the induced subgraph of $T$ on the defined vertex set)  rigorously defines a valid model of $\overline{\text{inf}(T)}$ strictly within $\overline{\text{inf}(T)}$. It is then relatively straightforward to apply Lemma~\ref{lem:Polat} to obtain the following fixed-element result.

\begin{theorem}\label{thm:fixed} Given a small tree $T$, there exists a vertex or edge of $T$ fixed by all $\mu \in \text{GME}(T)$.
\end{theorem}
\begin{proof}
Let $T$ be small and infinite, and set $T_1 = \overline{\text{inf}(T)}$. Observe that $T_1$ is connected, rayless, and since $T_1$ does not contain any end-points of $T$, $T_1$ is properly contained in $T$. Also, as observed previously, for any $\mu \in \text{GME}(T)$, $\mu\restriction T_1$ is a model of $T_1$ in $T_1$. For all $\alpha \in \text{Ord}$, construct $T_\alpha$ as follows:
$$
T_\alpha =
\begin{cases}
\overline{\text{inf}(T_\beta)} & \text{if } \alpha = \beta + 1, \\
\bigcap_{\beta < \alpha} T_\beta & \text{if } \alpha \text{ is a limit ordinal.}
\end{cases}
$$

Let $\kappa$ denote the smallest ordinal for which $(T_\alpha)$ is constant. This stabilization must occur when $T_\kappa = \emptyset$, because if $T_\kappa \neq \emptyset$, then $T_\kappa$ contains vertices of infinite degree, making $T_{\kappa +1} \neq T_\kappa$. In turn, by Lemma~\ref{lem:Polat}, we have that $\kappa = \lambda +1$ for some $\lambda \in \text{Ord}$, meaning $T_\lambda$ is non-empty, locally finite, and rayless. By design, for any $\mu\in \text{GME}(T)$, $\mu\restriction T_\lambda$ is a model of $T_\lambda$ in $T_\lambda$. By Theorem~\ref{thm:equiv}, this embedding is equivalent to an isomorphism of $T_\lambda$. It follows by Halin \cite{H} that some vertex or edge of $T_\lambda$ is fixed by all such $\mu$.
\end{proof}

We should remark that given a model $\mu \in \text{GME}(T)$, a ``fixed vertex'' $v \in V(T)$ strictly means $v \in V(\mu(v))$, rather than $\mu(v) = \{v\}$. The same subset logic applies to a fixed edge. This is merely a notational technicality and does not affect the reduction. Having established Theorem~\ref{thm:fixed}, the arguments employed for Theorem 1 in \cite{A} (or Theorem 4.2 in \cite{BP1}) apply directly to anchor the unrooted tree at this fixed element, effectively reducing the problem to the rooted case. Thus, we omit the full proof.

\begin{theorem} For any small tree $T$, TAC($\equiv^*$) holds. In turn, the Main Theorem holds for all small unrooted trees.
\end{theorem}

\section{Open Questions}
One way to naturally generalise this line of enquiry and extend TAC($\hat{R}$) is by asking about the size of an equivalence class of a tree up to a stronger relation. Given $\hat{R} \in \{\equiv, \equiv^\sharp, \equiv^*\}$ and $\hat{S} \in \{\cong, \equiv, \equiv^\sharp\}$, we write $[T]_{\hat{R}}/ \hat{S}$ for the equivalence class of $T$ with respect to $\hat{R}$ up to $\hat{S}$.

\vspace{0.2cm}
\noindent
\textbf{TAC($\hat{R},\hat{S}$).} For any tree $T$,
$$|[T]_{\hat{R}}/ \hat{S}| \notin \N\smallsetminus \{1\}.$$

\vspace{0.2cm}
\noindent
\textbf{RootedTAC($\hat{R},\hat{S}$).} For any tree $T$ and any $r\in V(T)$,
$$|[(T,r)]_{\hat{R}}/ \hat{S}| \notin \N\smallsetminus \{1\}.$$

\vspace{0.2cm}
Combining all of the results developed thus far, we obtain that TAC($\equiv,\cong$) is false, RootedTAC($\hat{R},\cong$) is true for any $\hat{R} \in \{\equiv, \equiv^\sharp, \equiv^*\}$, and TAC($\hat{R},\cong$) is true for any $\hat{R} \in \{\equiv^\sharp, \equiv^*\}$. We pose the following conjectures:

\vspace{0.2cm}
\noindent
\textbf{Conjecture 1.} TAC($\hat{R},\hat{S}$) holds for any $\hat{R} \in \{\equiv^\sharp, \equiv^*\}$ and $\hat{S} \in \{\cong, \equiv, \equiv^\sharp\}$.

\vspace{0.2cm}
\noindent
\textbf{Conjecture 2.} For any $\hat{R} \in \{\equiv, \equiv^\sharp, \equiv^*\}$ and $\hat{S} \in \{\cong, \equiv, \equiv^\sharp\}$:
\begin{enumerate}
\item RootedTAC($\hat{R},\hat{S}$) holds.
\item Restricted to small trees, TAC($\hat{R},\hat{S}$) also holds.
\end{enumerate}

Finally, the techniques developed and employed to prove Theorem~\ref{thm:large} in \cite{BP1} heavily rely on the fact that trees are \textit{well-quasi-ordered} (wqo) under the topological minor relation (and thus, also under the graph minor relation). Although Tyomkyn in \cite{T} establishes TAC for all rooted trees under the embeddability relation (which is not a wqo), there appears to be a strong connection between wqo and TAC for unrooted trees and other similar structures.

\end{document}